\documentclass[a4paper,oneside,reqno]{amsart}
\usepackage{geometry}
\usepackage{fancyhdr,mathrsfs,amssymb,amsmath,amsthm,latexsym}
\usepackage{mathtools}
\usepackage{enumerate,enumitem,csquotes,verbatim}
\usepackage{setspace}
\usepackage[numbers]{natbib}
\usepackage{hyperref}
\usepackage{youngtab}
\usepackage{ytableau}
\usepackage{color}
\usepackage{tikz}
\usepackage[table]{xcolor} 
\usepackage{soul}
\usepackage{multirow}
\usepackage{url}
\usepackage{ulem}
\usepackage{graphicx}
\usepackage{subcaption}
\usepackage{tabularx}

\usepackage{array}
\newcolumntype{Y}{>{\centering\arraybackslash}X}

\usepackage{ulem} 
\theoremstyle{plain}
\newtheorem{theorem}{Theorem}
\newtheorem{lemma}[theorem]{Lemma}

\newtheorem{corollary}[theorem]{Corollary}
\newtheorem{conjecture}[theorem]{Conjecture}

\theoremstyle{definition}

\theoremstyle{remark}

\newcommand{\A}{\mathcal{A}}
\newcommand{\B}{\mathcal{B}}
\newcommand{\C}{\mathcal{C}}
\newcommand{\D}{\mathcal{D}}
\newcommand{\E}{\mathcal{E}}

\title[]{Congruences via Partitions with Exactly Two Part Sizes}

\author{Sittinon Jirattikansakul}
\address{Department of Mathematics and Computer Science,
	Faculty of Science, Chulalongkorn University,
	Bangkok 10330, Thailand, and Centre of Excellence in Mathematics,
	Ministry of Higher Education, Science, Research and Innovation, Thailand}
\email{sittinon.j@chula.ac.th}

\author{Teeradej Kittipassorn}
\address{Department of Mathematics and Computer Science,
	Faculty of Science, Chulalongkorn University,
	Bangkok 10330, Thailand, and Centre of Excellence in Mathematics,
	Ministry of Higher Education, Science, Research and Innovation, Thailand}
\email{teeradej.k@chula.ac.th}

\author{Kraiwich Kongsiri}
\address{Department of Mathematics and Computer Science,
	Faculty of Science, Chulalongkorn University,
	Bangkok 10330, Thailand}
\email{6634303123@student.chula.ac.th}

\author{Nitipon Moonwichit}
\address{Institute of Mathematical Sciences,
	Claremont Graduate University,
	Claremont, CA, USA}
\email{nitipon.moonwichit@cgu.edu}

\author{Kirati Sriamorn}
\address{Department of Mathematics and Computer Science,
	Faculty of Science, Chulalongkorn University,
	Bangkok 10330, Thailand}
\email{kirati.s@chula.ac.th}

\begin{document}
	
	\begin{abstract}
		We prove the congruence $\sum_{1 \leq k < \sqrt{N}} \sigma_0 (N - k^2) \equiv 0 \pmod 4$, where $\sigma_0(m)$ denotes the number of positive divisors of $m$, for $N = An + B$ with $(A,B) \in \{ (16,14),$ $(36,30),$ $(72,42),$ $(196,70),$ $(252,114) \}$. Our proof relies on a result of Keith which states that $\nu_2 (N) \equiv 0 \pmod 4$, where $\nu_2(N)$ is the number of partitions of $N$ with exactly two part sizes. Inspired by Dewitt and Keith, our approach combines combinatorial arguments with modular arithmetic techniques. 
	\end{abstract}
	
	\maketitle


	\section{Introduction}

    Integer partitions play a central role in number theory and combinatorics. Since the foundational work of Euler~\cite{euler1797introductio}, partitions have been studied from a variety of perspectives, including generating functions by Andrews~\cite{andrews1998theory}, asymptotic analysis by Hardy and Ramanujan~\cite{hardy1918asymptotic}, and geometric interpretations via Ferrers and Young diagrams by MacMahon~\cite{macmahon1915combinatory}, among many others. 
	
	Define $\nu_k(N)$ to be the {\em number of partitions of $N$ in which exactly $k$ sizes of part appear}. For example, $\nu_2(4)=2$ since all partitions of 4 with exactly two part sizes are $(3,1)$ and $(2,1,1)$, and $\nu_3(4)=0$ since no partition of 4 has exactly three part sizes.
	
	This counting function has been studied by MacMahon~\cite{macmahon1921divisors}, Andrews~\cite{geandrew1999lattice}, and Tani and Bouroubi~\cite{tani2011enumeration} where the first two independently derived a compact formula for the case $k = 2$ as 
    \begin{align*}
        \nu_2(N) = \frac{1}{2}\left( \sum_{k=1}^{N-1} \sigma_0(k)\sigma_0(N-k) - \sigma_1(N) + \sigma_0(N) \right)
    \end{align*}
    where $\sigma_0(N)$ is the {\em number of positive divisors of $N$} and $\sigma_1(N)$ is the {\em sum of the positive divisors of $N$}. Keith~\cite{keith2017partitions} then later proved that $\nu_2(An+B)\equiv 0\pmod 4$ when $(A,B)$ is one of (16,14), (36,30), (72,42), (196,70) and (252,114). This inspires us to revisit several results of Keith~\cite{dewitt2025combinatorial} using a finer combinatorial approach based on multiset partition. In particular, we study configurations of Young diagrams of partitions of $N$ with exactly two part sizes in the arithmetic sequences considered by Keith~\cite{keith2017partitions}, and obtain the following congruence.
    
	
	\begin{theorem}\label{thm:sat}
		Let $(A,B)\in \{ (16,14), (36,30), (72,42), (196,70), (252,114) \}$ and $N = An+B$ for some nonnegative integer $n$. Then
		$$\sum\limits_{1\leq k<\sqrt{N}} \sigma_0(N-k^2) \equiv 0 \pmod4.$$
	\end{theorem}
	The sum $\sum_{1 \leq k < \sqrt{N}} \sigma_0(N - k^2)$ in Theorem~\ref{thm:sat} and its asymptotic behavior were investigated by Hooley~\cite{hooley1958representation} in his study of representations of $N$ as the sum of a square and a product where the number of such representations is equal to $\sum_{-\sqrt{N} < k < \sqrt{N}} \sigma_0 (N - k^2)$. 
	
	The following corollaries are consequences of Theorem~\ref{thm:sat}. 


	\begin{corollary}
		\label{corolldiv2}
		Let $N$ be a positive integer such that $N\equiv14\pmod{16}$ or $N\equiv70\pmod{196}$. Then the number of positive odd integers $k$ such that $k^2<N$ and $\sigma_0 (N - k^2) \equiv 2 \pmod 4$ is always even. 
	\end{corollary}

	\begin{corollary}
		\label{corolldiv3}
		Let $N$ be a positive integer such that $N\equiv B\pmod{A}$ where $(A,B)$ is one of $(36,30), (72,42),$ and $(252,114)$. Then $$\sum\limits_{\substack{1\leq k<\sqrt{N}\\3\nmid k}} \sigma_0(N-k^2) \equiv 0 \pmod4.$$
	\end{corollary}

To prove Theorem~\ref{thm:sat}, we first prove an auxiliary congruence relation in Theorem~\ref{doublecount} by assuming a weaker assumption on the natural number $N$.
    
	\begin{theorem}
		\label{doublecount}
		Let $N$ be a natural number satisfying
		\begin{itemize}
			\item $N$ cannot be written as a sum of two squares, and 
			\item $N=2m$ for some odd integer $m$. 
		\end{itemize}Then
		$$\nu_2(N)+\sum\limits_{1\leq k<\sqrt{N}} \sigma_0(N-k^2)+\frac{1}{2}\sigma_1(N) \equiv 0 \pmod4.$$
	\end{theorem}

    Note that according to the formula for $\nu_2(N)$ given in the introduction, 
    the congruence in Theorem~\ref{doublecount} can be written as 
    \begin{align*}
        \frac{1}{2} \left( \sigma_0(N) + \sum_{k=1}^{N-1} \sigma_0(k)\sigma_0(N-k) \right) + \sum_{1 \leq k < \sqrt{N}} \sigma_0 (N - k^2) \equiv 0 \pmod 4. 
    \end{align*}
    To the best of our knowledge, this congruence does not appear explicitly in the existing literature. 
	
	We organize the rest of the paper as follows. In Section~\ref{sec:bckgrnd}, we discuss background and terminology. Section~\ref{sec:thm:sat} is devoted to the proof of Theorem~\ref{thm:sat}, which serves as our main result. In Section~\ref{sec:otherresults}, we apply Theorem~\ref{thm:sat} and use some elementary number theory to prove Corollaries~\ref{corolldiv2} and~\ref{corolldiv3}. Finally, in Section~\ref{sec:conclude}, we conclude the paper with some remarks and a conjecture inspired by our findings.
	
	
	
	\section{Background}
	\label{sec:bckgrnd}
	
	
	Fix a natural number $n$. A {\em partition of $n$ with exactly one part size} is a tuple $\vec{v}=(n_1,n_1,\hdots, n_1)$ of $k$ copies of $n_1$ for some $k\geq 1$ satisfying $kn_1=n$.
	A {\em partition of $n$ with exactly two part sizes} is a tuple $\vec{v}=(n_1,n_1, \hdots, n_1,n_2,n_2, \hdots, n_2)$ where $n_1>n_2 \geq 1$ such that if $n_1$ appears $k_1$ times and $n_2$ appears $k_2$ times, then $k_1n_1+k_2n_2=n$. 
    In this paper, when we mention a {\em rectangle}, we mean a usual rectangular grid where each cell is a square. An {\em $m \times n$ rectangle} is a rectangle consisting of $m$ rows and $n$ columns. If $X$ is an $m \times n$ rectangle, denote by $X^T$ the $n \times m$ rectangle. 

	A {\em Young diagram for a partition} is the unique grid of cells where the cells are arranged in left-justified rows with the row lengths in non-increasing order. For each Young diagram consisting of $n$ cells in $k$ rows, we associate the diagram with the tuple $(n_1,n_2,\hdots, n_k)$ where $n_i$ is the number of cells lying in the $i$-th row. By this interpretation, each Young diagram in {\em rectangular shape} consisting of $n$ cells uniquely determines the partition of $n$ with exactly one part size. Likewise, each Young diagram in {\em L-shape} uniquely determines the partition of $n$ with exactly two part sizes. For example, Figure~\ref{young12}
is interpreted as the partition $(3,3,3,3)$ of $12$, and 
	Figure~\ref{young19} is interpreted as the partition $(5,5,5,2,2)$ of $19$.
	
	\begin{figure}
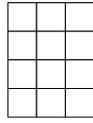
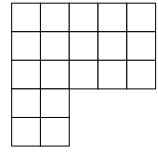

		\centering
		\begin{subfigure}{.45\linewidth}
			\centering
			\scriptsize\ydiagram[*(white)]{3,3,3,3}
			\caption{}
			\label{young12}
		\end{subfigure}
		\begin{subfigure}{.45\linewidth}
			\centering
			\scriptsize\ydiagram[*(white)]{5,5,5,2,2}
			\caption{}
			\label{young19}
		\end{subfigure}
		
		\caption{The Young diagrams for the partitions: (a) $(3,3,3,3)$ of $12$, and (b) $(5,5,5,2,2)$ of $19$}
	\end{figure}

   Any Young diagram corresponding to a partition with exactly two part sizes can be viewed as obtained by gluing two rectangles vertically, with the rectangle having the wider base placed on top. Conversely, every such L-shaped diagram can be decomposed in this way into a unique pair of rectangles. This observation motivates our approach and leads to a proof based on counting pairs of rectangles.

	\section{Proof of Theorem~\ref{thm:sat}}
	\label{sec:thm:sat}
    
    Let $N$ be a natural number that cannot be written as a sum of two squares. A {\em canonical pair for $N$} is an unordered pair of rectangles $\{X,Y\}$ such that the numbers of rows of $X$ and $Y$ are fewer than or equal to the numbers of columns of $X$ and $Y$, respectively, and the sum of numbers of cells of $X$ and $Y$ is $N$. Let $\A$ be the multiset consisting of 
    \[
    \{X,Y\}, \{X,Y^T\}, \{X^T,Y\}, \{X^T,Y^T\}
    \]
    for all canonical pairs $\{X,Y\}$ for $N$. We define subsets $\mathcal{B}, \mathcal{C}, \mathcal{D}$, and $\mathcal{E}$ of $\A$ as follows: 
    \begin{align*}
        \mathcal{B} &= \{\{X,Y\} \in \A \ | \ \text{the numbers of columns of $X$ and $Y$ are different}\}, \\
        \mathcal{C} &= \{\{X,Y\} \in \A \ | \ \text{$X$ or $Y$ is a square}\}, \\
        \mathcal{D} &= \{\{X,Y\} \in \A \ | \ \text{the numbers of columns of $X$ and $Y$ are the same}\}, \\
        \mathcal{E} &= \{\{X,Y\} \in \A \ | \ \text{$X = Y^T$}\}. 
    \end{align*}
    From the definitions above, there is a natural correspondence among the partitions with at most two part sizes, rectangular and L-shaped Young diagrams, and pairs of rectangles. As noted at the end of Section~\ref{sec:bckgrnd}, $\{X,Y\}$ corresponds to the Young diagram obtained by placing the rectangle with the wider base on top; in case where the bases are equal, either rectangle can be placed on top (resulting in the same diagram). This means $\A$ can be viewed as the multiset of Young diagrams obtained from all possible vertical gluings of two rectangles, where each rectangle may independently be replaced by its transpose. Consequently, the purpose of defining $\B, \C, \D$, and $\E$ is to classify those diagrams and count their multiplicities. We may interpret the sets as follows: 
    \begin{itemize}
        \item $\B$ is the set of all Young diagrams corresponding to the partitions of $N$ with exactly two part sizes, 
        \item $\C$ is the set of all Young diagrams corresponding to the partitions of $N$ that can be decomposed into a pair of rectangles $\{X,Y\}$ such that at least one of them is a square, 
        \item $\D$ is the multiset of Young diagrams with exactly one part size where we will later show how to count the multiplicities in Lemma~\ref{sizeBCDE}. 
        \item $\E$ is the set of all Young diagrams corresponding to the partitions of $N$ that can be decomposed into a pair of rectangles $\{X,Y\}$ such that $X = Y^T$. 
    \end{itemize}
    
    Denote by $\star$ the {\em gluing operation} and define $X \star Y$ to be the {\em multiset consisting of 
    \[
    \{X,Y\} , \{X,Y^T\} , \{X^T,Y\} , \{X^T,Y^T\}
    \]
    for any $\{X,Y\} \in \mathcal{A}$}. Here, $\{X,Y\}, \{X,Y^T\}, \{X^T,Y\}, \{X^T,Y^T\}$ can be viewed as the Young diagrams we mentioned above. See Figure~\ref{4adhesionsof4pairs} for an example. The figure demonstrates the four canonical pairs for $6$ together with their associated gluings. There are nine diagrams (with repetitions) in $\A$ corresponding to six elements in $\B$. Since $\B$ is a set, and among these nine diagrams there are three pairs of identical diagrams, each of these pairs is counted only once in $\B$. A similar argument applies to $\C, \D$, and $\E$. Therefore, at $N = 6$, 
    \[
    |\A| = 16 = 6 + 4 + 5 + 1 = |\B| + |\C| + |\D| + |\E|. 
    \]
    \begin{figure}
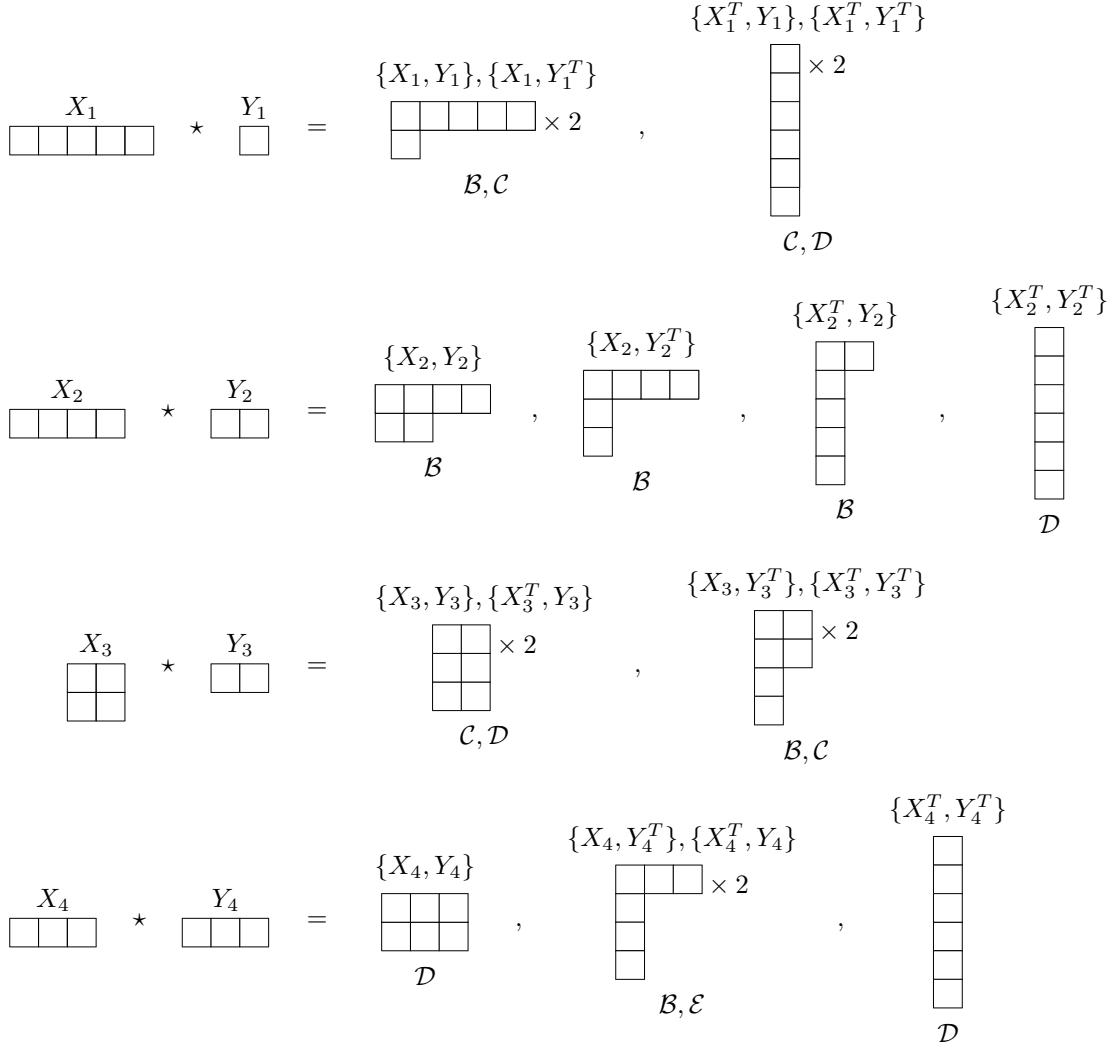

    \renewcommand{\arraystretch}{1.5}
		\begin{align*}
			{X_1 \atop \scriptsize\ydiagram{5}} \quad \star \quad {Y_1 \atop \scriptsize\ydiagram[*(white)]{1}} \quad &= \quad {\begin{array}{c} \{X_1,Y_1\}, \{X_1,Y_1^T\} \\ \scriptsize\ydiagram[*(white)]{5}*[*(white)]{5,1} \times 2 \\ {\mathcal{B},\mathcal{C}} \end{array}} \quad , \quad {\begin{array}{c} \{X_1^T,Y_1\},\{X_1^T,Y_1^T\} \\ \scriptsize\ydiagram[*(white)]{1,1,1,1,1}*[*(white)]{1,1,1,1,1,1} \times 2 \\ {\mathcal{C},\mathcal{D}} \end{array}} \\
			{X_2 \atop \scriptsize \ydiagram{4}} \quad \star \quad {Y_2 \atop \scriptsize\ydiagram[*(white)]{2}} \quad &= \quad {\begin{array}{c} \{X_2,Y_2\} \\ \scriptsize \ydiagram[*(white)]{4}*[*(white)]{4,2} \\ {\mathcal{B}} \end{array}} \quad , \quad {\begin{array}{c} \{X_2,Y_2^T\} \\ \scriptsize\ydiagram[*(white)]{4}*[*(white)]{4,1,1} \\ {\mathcal{B}}\end{array}} \quad , \quad {\begin{array}{c} \{X_2^T,Y_2\} \\ \scriptsize\ydiagram[*(white)]{0,1,1,1,1}*[*(white)]{2} \\ {\mathcal{B}}\end{array}} \quad , \quad {\begin{array}{c} \{X_2^T, Y_2^T\} \\ \scriptsize\ydiagram[*(white)]{1,1,1,1}*[*(white)]{0,0,0,0,1,1} \\ {\mathcal{D}}\end{array}} \\
			{X_3 \atop \scriptsize \ydiagram{2,2}} \quad \star \quad {Y_3 \atop \scriptsize\ydiagram[*(white)]{2}} \quad &= \quad {\begin{array}{c} \{X_3,Y_3\},\{X_3^T,Y_3\} \\ \scriptsize \ydiagram[*(white)]{2,2}*[*(white)]{0,0,2}  \times 2 \\ {\mathcal{C},\mathcal{D}}\end{array}} \quad , \quad {\begin{array}{c} \{X_3,Y_3^T\},\{X_3^T,Y_3^T\} \\ \scriptsize\ydiagram[*(white)]{2,2}*[*(white)]{0,0,1,1}  \times 2 \\ {\mathcal{B},\mathcal{C}}\end{array}} \\
			{X_4 \atop \scriptsize \ydiagram{3}} \quad \star \quad {Y_4 \atop \scriptsize\ydiagram[*(white)]{3}} \quad &= \quad {\begin{array}{c} \{X_4,Y_4\} \\ \scriptsize \ydiagram[*(white)]{3}*[*(white)]{0,3} \\ {\mathcal{D}}\end{array}} \quad , \quad {\begin{array}{c} \{X_4,Y_4^T\},\{X_4^T,Y_4\} \\ \scriptsize\ydiagram[*(white)]{3}*[*(white)]{0,1,1,1} \times 2 \\ {\mathcal{B,E}}\end{array}} \quad , \quad {\begin{array}{c} \{X_4^T,Y_4^T\} \\ \scriptsize\ydiagram[*(white)]{1,1,1}*[*(white)]{0,0,0,1,1,1} \\ {\mathcal{D}} \end{array}}
		\end{align*}
		\caption{Young diagrams and their corresponding canonical pairs for $N = 6$}
		\label{4adhesionsof4pairs}
	\end{figure}
    
    Note that any two multisets $X \star Y$ and $Z \star W$
    are either identical or disjoint. Let $m_\mathcal{S}(A)$ be the {\em multiplicity of $A$ in a multiset $\mathcal{S}$}. Let $\{X,Y\} \in \A$. Since $N$ is not a sum of two squares, we know that if $X = Y^T$, then $X$ and $Y$ are not squares. Hence, $m_\A(\{X,Y\}) \leq 2$. In particular, 
    \[
    m_\mathcal{A}(\{X,Y\}) = \begin{cases}
        1 & \text{if $X$ and $Y$ are not squares and $X \neq Y^T$,} \\
        2 & \text{if $X$ or $Y$ is a square or $X = Y^T$.} 
    \end{cases}
    \]

        \begin{lemma}
        \label{countingmultisets}
			$|\mathcal{A}|=|\mathcal{B}| + |\mathcal{C}| + |\mathcal{D}| + |\mathcal{E}|$. 
		\end{lemma}

        \begin{proof}[Proof of Lemma~\ref{countingmultisets}]

            We will show that for any $\{X,Y\} \in \mathcal{A}$, 
            \[
            m_\mathcal{A}(\{X,Y\}) = m_{\B}(\{X,Y\}) + m_{\C}(\{X,Y\}) + m_{\D}(\{X,Y\}) + m_{\E}(\{X,Y\}). 
            \]
            \begin{itemize}
                \item Case 1: the numbers of columns of $X$ and $Y$ are different. Then $m_\B(\{X,Y\}) = 1$ and $m_\D(\{X,Y\}) = 0$. 
                \begin{itemize}
                    \item Case 1.1: $X \neq Y^T$. Then $m_\E(\{X,Y\}) = 0$. 
                    \begin{itemize}
                        \item Case 1.1.1: $X$ and $Y$ are not squares. We have $m_\A(\{X,Y\}) = 1$ and $m_\C(\{X,Y\}) = 0$. 
                        \item Case 1.1.2: $X$ or $Y$ is a square. We have $m_\A(\{X,Y\}) = 2$ and $m_\C(\{X,Y\}) = 1$. 
                    \end{itemize}
                    \item Case 1.2: $X = Y^T$. Then $m_\A(\{X,Y\}) = 2$ and $m_\E(\{X,Y\}) = 1$. Since $N$ is not a sum of two squares, we have that $X$ and $Y$ both are not squares. Therefore, $m_\C(\{X,Y\}) = 0$. 
                \end{itemize}
                \item Case 2: the numbers of columns of $X$ and $Y$ are the same. Then $m_\B(\{X,Y\}) = 0$ and $m_\D(\{X,Y\}) = 1$. Since $N$ is not a sum of two squares, we know that $X \neq Y^T$. Hence, $m_\E(\{X,Y\}) = 0$. 
                \begin{itemize}
                    \item Case 2.1: $X$ and $Y$ are not squares. Then $m_\A(\{X,Y\}) = 1$ and $m_\C(\{X,Y\}) = 0$. 
                    \item Case 2.2: $X$ or $Y$ is a square. We have $m_\A(\{X,Y\}) = 2$ and $m_\C(\{X,Y\}) = 1$. 
                \end{itemize}
            \end{itemize}
\begin{table}[h!]
\centering
\renewcommand{\arraystretch}{1.5}
\begin{tabular}{|c|cc|c|c|c|c|c|}
\hline
the numbers of columns of $X,Y$ & \multicolumn{2}{c|}{} & $m_\A$ & $m_\B$ & $m_\C$ & $m_\D$ & $m_\E$ \\ \hline\hline
\multirow{3}{*}{different} & \multicolumn{1}{c|}{\multirow{2}{*}{$X \neq Y^T$}} & \begin{tabular}[c]{@{}c@{}}$X,Y$ are \\ not squares\end{tabular} & 1 & 1 & 0 & 0 & 0 \\ \cline{3-8} 
& \multicolumn{1}{c|}{} & \begin{tabular}[c]{@{}c@{}}$X$ or $Y$ is \\ a square\end{tabular} & 2 & 1 & 1 & 0 & 0 \\ \cline{2-8} 
& \multicolumn{2}{c|}{\begin{tabular}[c]{@{}c@{}}$X = Y^T$\\ (then $X,Y$ are not squares)\end{tabular}} & 2 & 1 & 0 & 0 & 1 \\ \hline
\multirow{2}{*}{\begin{tabular}[c]{@{}c@{}}same\\ (then $X \neq Y^T$)\end{tabular}} & \multicolumn{2}{c|}{\begin{tabular}[c]{@{}c@{}}$X,Y$ are \\ not squares\end{tabular}} & 1 & 0 & 0 & 1 & 0 \\ \cline{2-8} 
& \multicolumn{2}{c|}{\begin{tabular}[c]{@{}c@{}}$X$ or $Y$ is \\ a square\end{tabular}} & 2 & 0 & 1 & 1 & 0 \\ \hline
\end{tabular}
\caption{The multiplicities of $\{X,Y\}$ in $\A, \B, \C, \D$, and $\E$}
\end{table}
Therefore, $|\A| = |\B| + |\C| + |\D| + |\E|$. 
\end{proof}

        
        
        We then determine the sizes of $\mathcal{B}$, $\mathcal{C}$, $\mathcal{D}$, and $\mathcal{E}$.
		

        \begin{lemma}
        \label{sizeBCDE}
            Assume that a natural number $N$ is not a sum of two squares. Then, 

            \begin{align*}
			|\mathcal{B}| &= \nu_2(N), \\
			|\mathcal{C}| &= \sum_{1 \leq k < \sqrt{N}} \sigma_0(N - k^2), \\
			|\mathcal{D}| &= \sum_{d \mid N} \left\lfloor \frac{d}{2} \right\rfloor, \\
			|\mathcal{E}| &= \frac{1}{2} \sum_{d \mid \frac{N}{2}} 1. 
		\end{align*}
        
        \end{lemma}
		
		\begin{proof}[Proof of Lemma~\ref{sizeBCDE}] 
            It is obvious that $|\B|$ is $\nu_2(N)$, the number of partitions of exactly two part sizes. Denote by $R(m,n)$ the rectangle with $m$ rows and $n$ columns. For each $\{X,Y\} \in \C$, we may assume that $x = R(k,k)$ for some positive integer $k < \sqrt{N}$. Hence, the number of rectangles $Y$ is the number of all positive divisors of $N - k^2$, which is $\sigma_0(N-k^2)$. Thus, 
            \[
            |\C| = \sum_{1 \leq k < \sqrt{N}} \sigma_0(N - k^2). 
            \]
            By the definition of $\D$, for each $\{X,Y\} \in \D$, we may assume that the number of rows of $X$ is greater than or equal to the number of rows of $Y$. Hence, $X = R(d-i, \frac{N}{d}), Y = R(i, \frac{N}{d})$ for some $d$ and $i$ with $d \mid N$ and $1 \leq i \leq \lfloor \frac{d}{2} \rfloor$. Therefore, 
            \[
            |\D| = \sum_{d \mid N} \left\lfloor \frac{d}{2} \right\rfloor. 
            \]
            For each $\{X,Y\} \in E$, we have that $X = X_d = R(d, \frac{N}{2d})$ and $Y = X_d^T$, for some positive integer $d \mid \frac{N}{2}$. Note that $\{X_d,X_d^T\} = \{X_{N/2d},X_{N/2d}^T\}$ and $d \neq \frac{N}{2d}$. We have 
            \[
            |\E| = \frac{1}{2}\sum_{d \mid \frac{N}{2}} 1. \qedhere
            \]
        \end{proof}

        Also note further that if $N$ is $2m$ for some odd $m$ and observe that $\sigma_0(n)=\sum_{d \mid n} 1$ and $\sigma_1(n)=\sum_{d \mid n} d$, we can simplify $|\mathcal{D}|$ and $|\mathcal{E}|$ as  
		\begin{align*}
			|\mathcal{D}| = \sum_{d \mid N} \left\lfloor \frac{d}{2} \right\rfloor &= \sum_{d \mid N, 2 \nmid d} \frac{d-1}{2} + \sum_{d \mid N, 2 \mid d} \frac{d}{2} \\
			&= \sum\limits_{d \mid N, 2\nmid d} \frac{d}{2}+\sum\limits_{d \mid N, 2 \mid d} \frac{d}{2}-\dfrac{1}{2}\sum\limits_{d \mid N,2 \nmid d}1 \\
			&= \frac{1}{2}\sigma_1(N) - \frac{1}{2}\sigma_0\left(\frac{N}{2}\right), \\
			|\mathcal{E}| = \frac{1}{2}\sum\limits_{d \mid \frac{N}{2}} 1&= \frac{1}{2}\sigma_0\left(\frac{N}{2}\right). 
		\end{align*}


        \begin{proof}[Proof of Theorem~\ref{doublecount}]
		By Lemma~\ref{countingmultisets} and Lemma~\ref{sizeBCDE}, we have 
		\[
		|\A| = |\mathcal{B}| + |\mathcal{C}| + |\mathcal{D}| + |\mathcal{E}| = \nu_2(N) + \sum_{1 \leq k < \sqrt{N}} \sigma_0(N-k^2) + \frac{1}{2}\sigma_1(N). 
		\]
        By the definition of a $\A$, it is clear that $|\mathcal{A}| \equiv 0 \pmod 4$. 
		\end{proof}


        
		We are now ready to prove Theorem~\ref{thm:sat}.
		\begin{proof}[Proof of Theorem~\ref{thm:sat}]
		Assume $N$ satisfies the conditions as in Theorem~\ref{thm:sat}. Notice that these choices of $N=An+B$ cannot be written as a sum of two squares (which can be proven by evaluating all possible sums of two squares in modulo $A$).  From Theorem~\ref{doublecount}, it provides that
		$$\nu_2(N)+\sum\limits_{1 \leq k < \sqrt{N}} \sigma_0(N-k^2)+\frac{1}{2}\sigma_1(N) \equiv 0 \pmod4.$$
		Due to the fact that these choices of $N$ and~\cite{keith2017partitions} imply $\nu_2(N)\equiv0\pmod4$, it suffices to show that the sum of divisors of $N$ is divisible by 8.
		
		For the case $N=16n+14$, assume that
		$$N=2\prod_{\alpha=1}^{m_1}p_{1\alpha} \prod_{\beta=1}^{m_3}p_{3\beta} \prod_{\gamma=1}^{m_5}p_{5\gamma} \prod_{\delta=1}^{m_7}p_{7\delta}$$
		where $p_{ri}$ is a prime factor of $N$ which is congruent to $r$ in modulo 8 for all $r=1,3,5,7$.
		Now, consider the fact that
		\[
		7\equiv N/2 \equiv 1^{m_1}\cdot3^{m_3}\cdot5^{m_5}\cdot7^{m_7}\equiv3^{m_3}\cdot5^{m_5}\cdot7^{m_7}\pmod8.
		\]
        Since $3^2\equiv5^2\equiv7^2\equiv1\pmod8$, we can compute as follows all parity possibilities for $m_3, m_5$, and $m_7$.
		\[
		\begin{split}
			3^0 5^0 7^0 \equiv 1 \pmod8, \hspace{5mm}
			& 3^1 5^0 7^0 \equiv 3 \pmod8,\\
			3^0 5^1 7^0 \equiv 5 \pmod8, \hspace{5mm}
			& 3^1 5^1 7^0 \equiv 7 \pmod8, \\
			3^0 5^0 7^1 \equiv 7 \pmod8, \hspace{5mm}
			& 3^1 5^0 7^1 \equiv 5 \pmod8, \\
			3^0 5^1 7^1 \equiv 3 \pmod8, \hspace{5mm}
			& 3^1 5^1 7^1 \equiv 1 \pmod8.
		\end{split}
		\]
        
		
		For the case where $m_7$ is odd, there exists a prime factor $p$ of $N$ and a positive odd integer $a$ such that $N$ is divisible by $p^a$ but not by $p^{a+1}$, and the remainder of $p$ divided by 8 is 7. Notice that $\sigma_1(N)$ is divisible by $1+p+p^2+\cdots+p^{a}$, and
		$$1+p+p^2+\cdots+p^{a}\equiv1+7+7^2+\cdots+7^{a}\equiv0\pmod8.$$
		Thus, the quantity $\sigma_1(N)$ is divisible by 8.
		
		For the case where $m_7$ is even ($m_3$ and $m_5$ are odd), there exist prime factors $p$ and $q$ of $N$ and positive odd integers $a$ and $b$ such that $N$ is divisible by $p^a$ and $q^b$ but not by $p^{a+1}$ and $q^{b+1}$, and the remainders of $p$ and $q$ divided by 8 are 3 and 5 respectively. Since $\sigma_1(N)$ is divisible by $(1+p+p^2+\cdots+p^{a})(1+q+q^2+\cdots+q^{b})$, and
		\[\begin{split}
			1+p+p^2+\cdots+p^{a}\equiv1+3+3^2+\cdots+3^{a}\equiv0\pmod4,\\
			1+q+q^2+\cdots+q^{b}\equiv1+5+5^2+\cdots+5^{b}\equiv0\pmod2,
		\end{split}\]
		the quantity $\sigma_1(N)$ is also divisible by 8. Hence, for $N=16n+14$, we have that $\sigma_1(N)$ is also divisible by 8.
		
		For the case $N=6k$ for some positive integer $k$ such that $k$ is not a square and $k$ and 6 are co-prime (this includes the cases when $(A,B)\in\{(36,30),(72,42),(252,114)\}$), since $k$ is not a square, there exists a prime factor $p$ of $k$ (which is also a prime factor of $N$) and a positive odd integer $a$ such that $k$ is divisible by $p^a$ but not $p^{a+1}$. Since $3p^{a}$ is a factor of $N$ such that $N$ is not divisible by $3^2$ and $p^{a+1}$, the quantity $\sigma_1(N)$ is divisible by
		$$(1+3)(1+p+p^2+\cdots+p^{a})=4(1+p+p^2+\cdots+p^{a}).$$
		Notice that
		$$1+p+p^2+\cdots+p^{a}\equiv1+1^1+1^2+\cdots+1^{a}\equiv0\pmod2.$$
		Thus, the quantity $\sigma_1(N)$ is divisible by 8.
		
		For the case $N=196n+70$, we have $N=7(28n+10)$. Notice that 7 and $28n+10$ are co-prime. Thus, we get that $N$ is divisible by 7 but not $7^2$. So, the quantity $\sigma_1(N)$ is divisible by $1+7=8$.
        \end{proof}
		
		\section{Proofs of Corollaries}
		\label{sec:otherresults}

		\begin{proof}[Proof of Corollary~\ref{corolldiv2}]
			Consider 
			\begin{align*}
				\sum\limits_{1\leq k<\sqrt{N}} \sigma_0(N-k^2) &= \sum_{\substack{1\leq k<\sqrt{N}\\k\text{ is odd}}} \sigma_0(N-k^2) +  \sum_{\substack{1\leq k<\sqrt{N}\\k\text{ is even}}} \sigma_0 (N - k^2).
			\end{align*}
			For the case $N\equiv 14\pmod{16}$, let $n$ be a nonnegative integer such that $N=16n+14$ and $l$ be an integer. Consider the quantity $N-(2l)^2=16n+14-4l^2=2(8n+7-2l^2)$.  Observe that any square of an integer is congruent to $0,1,$ or $4$ mod $8$. Therefore, $8n+7-2l^2$ is congruent to $5$ or $7$ mod $8$. Hence, $8n+7-2l^2$ is odd and cannot be a square. Thus, the number of divisors of $8n+7-l^2$ is even, and the number of divisors of $N-(2l)^2$ is
			$$\sigma_0(N-(2l)^2)=2\sigma_0(8n+7-l^2)\equiv 0\pmod4.$$
			For the case $N\equiv 70\pmod{196}$, let $m$ be a nonnegative integer such that $N=196m+70$ and $r$ be an integer. Consider the quantity $N-(2r)^2=196m+70-4r^2=2(98m+35-2r^2)$. Suppose that $98m+35-2r^2$ is a square. Observe that any square is congruent to $0,1,2,$ and $4$ mod $7$. Thus, $-2r^2$ is congruent to one of $0,3,5,$ and $6$ mod $7$. Since $98m+35-2r^2$ is a square, $r$ is congruent to $0$ mod $7$. Let $s$ be an integer such that $r=7s$. Then we get that
			$$98m+35-2r^2=98m+35-2(7s)^2=7(14m+5-14s^2).$$
			Notice that $98m+35-2r^2$ is divisible by $7$ but not by $7^2$. Hence, $98m+35-2r^2$ is not a square, leading to a contradiction. Therefore, $98m+35-2r^2$ is not a square and also odd. Thus, the number of divisors of $98m+35-2r^2$ is even, and the number of divisors of $N-(2r)^2$ is
			$$\sigma_0(N-(2l)^2)=2\sigma_0(98m+35-2r^2)\equiv 0\pmod4.$$
			From these two cases, we get that $\sigma_0(N-(2l)^2)$ is divisible by $4$, and
			\[\begin{split}
				\sum\limits_{1\leq k<\sqrt{N}} \sigma_0(N-k^2)
				&= \sum_{\substack{1\leq k<\sqrt{N}\\k\text{ is odd}}} \sigma_0(N-k^2) +  \sum_{\substack{1\leq k<\sqrt{N}\\k\text{ is even}}} \sigma_0 (N - k^2),\\
				&\equiv \sum_{\substack{1\leq k<\sqrt{N}\\k\text{ is odd}}} \sigma_0(N-k^2)\pmod4.
			\end{split}\]
			Due to the fact that $N$ cannot be written as a sum of two squares, the number of divisors of $N-k^2$ is even for all integer $k$. By Theorem~\ref{thm:sat}, we get that the amount of odd positive integer $k$ such that $k^2<N$ and $\sigma_0 (N - k^2) \equiv 2 \pmod 4$ is even.
		\end{proof}
		
		\begin{proof}[Proof of Corollary~\ref{corolldiv3}]
			Observe that $N=6n$ for some integer $n$ such that $6$ and $n$ are co-prime. We will show that $\sigma_0(N - (3a)^2)$ is a multiple of $4$ for all integer $a$ with $N > (3a)^2$. Suppose that there exist positive integers $a$ and $b$ such that $b = N - (3a)^2$ and $\sigma_0(b)$ is not a multiple of $4$. Since $N$ is divisible by $3$ but not by $3^2$, there exists a positive integer $b'$ such that $b=3b'$ and both $3$ and $b'$ are co-prime. Notice that $\sigma_0(b)=\sigma_0(3b')=2\sigma_0(b')$ is not a multiple of $4$. So, $\sigma_0(b')$ is odd, meaning that there exists an integer $c$ such that $b'=c^2$. Thus, we get that $6n=N=(3a)^2+b=9a^2+3c^2$, implying that $2n=3a^2+c^2$. Since $6$ and $n$ are co-prime, we have $3a^2+c^2=2n\equiv2\pmod4$, contradicting that $3a^2+c^2$ is congruent to one of $0,1,$ and $3$ mod $4$. Hence, if $N>(3a)^2$, then $\sigma_0(N-(3a)^2)$ is a multiple of $4$ for all integer $a$. Therefore,
			$$\sum\limits_{\substack{1\leq k<\sqrt{N}\\3\nmid k}} \sigma_0(N-k^2) \equiv\sum\limits_{1\leq k<\sqrt{N}} \sigma_0(N-k^2) \equiv 0 \pmod4$$
			by Theorem~\ref{thm:sat}.
		\end{proof}

		\section{Concluding Remarks}
		\label{sec:conclude}
		
		As shown in the proof of our main result, Theorem~\ref{thm:sat}, the theorem currently requires rather restrictive conditions on $N$, i.e., $N = An + B$ where $(A,B)$ belongs to the set $$\{ (16,14), (36,30), (72,42), (196,70), (252,114) \}.$$ Although this constraint is sufficient for our argument, we performed computational experiments for small values of $N$ (less than 100,000) which suggests a consistent trend indicating that the theorem may also hold if $N = 8n+6$. We can see that the congruence class $N \equiv 6 \pmod 8$ has the same structural features as the specific pairs $(A,B)$ appearing in Theorem~\ref{thm:sat} in the sense that any such $N$ is automatically not a sum of two squares, since an integer of the form $a^2 + b^2$ cannot be congruent to $6$ mod $8$, and $N = 2m$ for some odd integer $m$. However, the pair $(A,B) = (8,6)$ differs from the other pairs because there exist integers $N = 8n+6$ such that $\frac{1}{2}\sigma_1(N) \equiv 2 \pmod 4$, and $\nu_2(N) \equiv 2 \pmod 4$. This does not happen when $n$ is an odd number, but $n$ being odd means $n = 2m+1$ for some $m$ and $N = 8(2m+1) + 6 = 16m+14$ which corresponds to the pair $(A,B) = (16,14)$ in Theorem~\ref{thm:sat}. Therefore, it remains to consider the case when $n$ is even, which we believe satisfies the desired property. 
		
		\begin{conjecture}
			Let $N = 16n+6$ for some nonnegative integer $n$. Then 
			\[
			\sum_{1 \leq k < \sqrt{N}} \sigma_0(N - k^2) \equiv 0 \pmod 4. 
			\]
		\end{conjecture}
		
		Another computational observation is the following: among the pairs $(A,B)$ for which 
		\[
		\sum_{1 \leq k < \sqrt{N}} \sigma_0(N-k^2) \equiv 0 \pmod 4
		\]
		for all $N = An + B$ in our search range, the coefficient $A$ (less than $1,000$) always appears to be divisible by $4$, and the coefficient $B$ is always divisible by $2$ but not $4$. 
        However, this phenomenon is not explained by our approach in this work or in previous mentioned works. We make another conjecture. 
		
		\begin{conjecture}
			Let $A$ and $B$ be nonnegative integers such that $0\leq B<A$. If $N=An+B$ satisfies
			\[
			\sum_{1 \leq k < \sqrt{N}} \sigma_0(N - k^2) \equiv 0 \pmod 4 
			\]
			for all nonnegative integer $n$ then $A \equiv 0 \pmod 4$ and $B \equiv 2 \pmod 4$.
		\end{conjecture}
		
		\section*{Acknowledgement}
		
		We would like to thank Chatchawan Panraksa, Napon Apintanaphong, and Kornchawan Tantivisethsak for helpful discussion at the Geometry and Combinatorics Boot Camp 2025, Bangkok, Thailand. We are also grateful to the Department of Mathematics and Computer Science, Faculty of Science, Chulalongkorn University for hosting the boot camp, which provided the opportunity for us to meet and collaborate. 
		
		\bibliographystyle{siam}
		\bibliography{partsat}

	\end{document}